\documentclass[a4paper,10pt]{article}
\usepackage[utf8]{inputenc}
\usepackage{amsmath}
\usepackage{amsthm}
\usepackage{amsfonts}
\usepackage{amssymb}
\usepackage{mathrsfs}
\usepackage{comment}
\usepackage{geometry}
\usepackage{tikz}
\usetikzlibrary{automata,positioning}
\title{Explicit Implicit Function Theorem for All Fields}
\author{YINING HU \\
CNRS, Institut de Math\'ematiques de Jussieu-PRG \\
Universit\'e Pierre et Marie Curie, Case 247 \\
4 Place Jussieu \\
F-75252 Paris Cedex 05 (France) \\
{\tt yining.hu@imj-prg.fr}}
\date{}
\begin{document}

\maketitle

\begin{abstract}
We give an explicit implicit function theorem for formal power series that is valid for all fields, which implies in particular Lagrange inversion formula and
and Flajolet-Soria coefficient extraction formula known for fields of characteristic 0.
\end{abstract}
\newtheorem{thm}{Theorem}
\newtheorem{prop}{Proposition}
\newtheorem{coro}{Corollary}
\theoremstyle{definition}
\newtheorem{defi}{Definition}
\newtheorem{remark}{Remark}
\begin{thm}\label{main}
Let $K$ be an arbitrary field. If $P(X,Y)\in K[[X,Y]]$ and $f(X)\in K[[X]]$ are such that $f(0)=0$, $P(X,f(X))=f(X)$ and $P'_Y(0,0)=0$.
 Then $$[X^n]f = \sum _{m\geq 1} [X^n Y^{m-1}](1-P'_Y(X,Y))P^m(X,Y).$$
 If the characteristic of $K$ is $0$, we also have the following form
 $$[X^n]f=\sum _{m\geq 1}\frac{1}{m}[X^n Y^{m-1}]P^m(X,Y).$$
\end{thm}

\begin{remark}
The conditions $P(X,f(X))=0$ and $f(0)=0$ imply that $P(0,0)=0$. As $P'_Y(0,0)$ is also $0$, the sums in both expressions of $[X^n]f$ are finite.
\end{remark}
\begin{remark}
When $P(X,Y)=X\phi(Y)$, where $\phi(X)\in K[[X]]$ and $\phi(0)\neq 0$, we obtain the Lagrange inversion formula
$$[X^n]f=[Y^{n-1}](\phi(X)^n-Y\phi'(X)\phi(X)^{n-1}).$$
If the characteristic of $K$ is $0$, we also have the following form
$$[X^n]f=\frac{1}{n}[Y^{n-1}]\phi(Y)^n.$$
\end{remark}
\begin{remark}
When $P(X,Y)$ is a polynomial in $X$ and $Y$, we obtain a generalisation of Flajolet-Soria coefficient extraction formula \cite{soria}. 
\end{remark}

\begin{defi}
 Let $P(X,Y)\in K[[X,Y]]$, 
 $$P(X,Y)=\sum\limits_{j=0}^{\infty}a_{j}(x)Y^j,  $$
 where $a_j(x)\in K[[X]]$.
 We define 
 
$$P^{[m]}(X,Y)=\sum\limits_{j=m}^{\infty}\binom{j}{m}a_j(x)Y^{j-m}. $$
 for $m\in\mathbb{N}$.
  \end{defi}
The motivation of the definition of $P^{[m]}$ is to avoid the factorials in the denominators in the Taylor series, which do not make sense in positive characteristic.
Once this obstacle is circumvented, the Taylor formula works as expected.

\begin{prop}\label{taylor}
Let $K$ be an arbitrary field. Let $P(X,Y)\in K[[X,Y]]$ and $f(X)\in K[[X]]$ with  $f(0)=0$. Then 
\begin{equation}
 P(X,Y)=\sum\limits_{m=0}^{\infty}(Y-f(X))^m P^{[m]}(X,f(x)).\tag{*}
\end{equation}

\end{prop}
\begin{proof}
Let $P(X,Y)=\sum\limits_{j=0}^{\infty} a_j(X) Y^j$ with $a_j(X)\in K[[X]]$ for $j\in\mathbb{N}$. We prove that for all $k\in\mathbb{N}$,
the coefficient of $Y^k$ in the left side and right side of $(*)$ is equal. Indeed, we have
\begin{align*}
 &[Y^k]\sum\limits_{m=0}^{\infty}(Y-f(X))^m P^{[m]}(X,f(x))\\
 =&\sum\limits_{m=k}^{\infty}\binom{m}{k}(-f(X))^{m-k}\sum\limits_{j=m}^{\infty}\binom{j}{m}a_j(x)f(X)^{j-m}\\
 =&\sum\limits_{j=k}^{\infty}a_j(X)f(X)^{j-k}\sum\limits_{m=k}^{j}\binom{j}{m}\binom{m}{k}(-1)^{m-k}\\
 =&\sum\limits_{j=k}^{\infty}a_j(X)f(X)^{j-k}\sum\limits_{m=k}^{j}\binom{j}{j-m,m-k,k}(-1)^{m-k}\\
 =&a_k(X).
\end{align*}
We have the last equality because $\sum\limits_{m=k}^{j}\binom{j}{j-m,m-k,k}(-1)^{m-k}=1$ if $j=k$ and $0$ if $j>k$. This is because we have the multinomial
expansion
\begin{align*}
 (a+b+c)^j&=\sum\limits_{k\leq m\leq j}\binom{j}{j-m,m-k,k}a^{j-m}b^{m-k}c^k\\
 &=\sum\limits_{k=0}^j\sum\limits_{m=k}^{j}\binom{j}{j-m,m-k,k}a^{j-m}b^{m-k}c^k\\
\end{align*}
When we take $a=1$, $b=-1$, the identity becomes
$$c^j=\sum\limits_{k=0}^j\left(\sum\limits_{m=k}^{j}\binom{j}{j-m,m-k,k}(-1)^{m-k}\right)c^k.$$
Seeing this as a polynomial identity in the variable $c$ gives us the desired result.
\begin{comment}
 We first prove the equality for $P(X,Y)\in K[[X]][Y]$.  Let 
 $$P(X,Y)=\sum\limits_{j=0}^{d}a_j(X)Y^j.$$
In this case $P^{[m]}(X,Y)=\sum\limits_{j=m}^{d} \binom{j}{m} a_j(X)Y^{j-m}$ for $m\leq d$ and $P^{[m]}(X,Y)=0$ for $m>d$. Therefore the left side of the $(*)$ becomes
 \begin{align*}
  &\sum\limits_{m=0}^{d}(Y-f(X))^m P^{[m]}(X,f(x))\\
  =&\sum\limits_{m=0}^{d}(Y-f(X))^m\sum\limits_{j=m}^{d}\binom{j}{m}a_j(X)f(X)^{j-m}\\
  =&\sum\limits_{j=0}^{d}a_j(x)\sum\limits_{m=0}^{j}\binom{j}{m}(Y-f(X))^mf(X)^{j-m}\\
  =&\sum\limits_{j=0}^{d}a_j(x)Y^j\\
  =&P(X,Y).
 \end{align*}
When $P(X,Y)$ is in $K[[X,Y]]$, the left side of $(*)$ is well defined as $f(0)=0$. When we truncate both side of $(*)$ at $Y^d$, we have equality because of the reasoning above.
As this is true for all $d\in\mathbb{N}$, we have equality $(*)$.
\end{comment}
\end{proof}
The following corollary is immediate.
\begin{coro}\label{factorize}
 Let $K$ be an arbitrary field. Let $Q(X,Y)\in K[[X,Y]]$ and $f(X)\in K[[X]]$ be such that $f(0)=0$ and $Q(X,f(X))=0$. Then there exists $R(X,Y)\in K[[X,Y]]$
 such that $Q(X,Y)=(Y-f(X))R(X,Y)$.
\end{coro}

\begin{defi}
For the formal power series in $K ((X_1,...,X_m)) $
$$P(X_1,X_2,...,X_m)=\sum_{n_i>-\mu}a_{n_1n_2...n_m}X_1^{n_1}X_2^{n_2}...X_m^{n_m}$$
its (principal) diagonal $\mathscr{D}f(t)$ is defined as the element in $\kappa((T)) $
$$\mathscr{D}P(T)=\sum a_{nn...n}T^n.$$
\end{defi}

The following proposition is a generalization of Proposition 2 from \cite{fu}, the only difference in the proof is in the first step where we use Corollary
\ref{factorize} to factorize $Q(X,Y)$.

\begin{prop}\label{furst}
 Let $K$ be an arbitrary field. Let $Q(X,Y)\in K[[X,Y]]$ and $f(X)\in K[[X]]$ be such that $f(0)=0$, $Q(X,f(X))=0$ and $Q'_Y(0,0)\neq 0$, then
 $$f(X)=\mathscr{D}\left(Y^2 \frac{Q'_Y(XY,Y)}{Q(XY,Y)}\right).$$
\end{prop}
\begin{proof}
 Using Corollary \ref{factorize} we can write $Q(X,Y)=(Y-f(X))R(X,Y)$ with $R(X,Y)\in K[[X,Y]]$. We have $R(0,0)\neq 0$ because $f(0)=0$ and $Q'_Y\neq 0$. Then
 $$\frac{1}{Q(X,Y)}Q'_Y(X,Y)=\frac{1}{Y-f(X)}+\frac{R'_Y(X,Y)}{R(X,Y)}.$$
 Replacing $X$ by $XY$ and multiplying by $Y^2$ we get
 \begin{equation}\mathscr{D}\left(Y^2\frac{Q'_Y(XY,Y)}{Q(XY,Y)}\right)=\mathscr{D}\left(\frac{Y^2}{Y-f(XY)}\right)+\mathscr{D}\left(Y^2\frac{R'_Y(XY,Y)}{R(XY,Y)}\right).\tag{\dag}\end{equation}                                                                                                                                                                                 
 For the first term on the right side of (\dag)  we have 
 \begin{align*}
  &\mathscr{D}\left(\frac{Y^2}{Y-f(XY)}\right)\\=&\mathscr{D}\left(\frac{Y}{1-Y^{-1}f(XY)}\right)\\
  =&\mathscr{D}\left(\ \sum\limits_{n=0}^{\infty}Y^{-n+1} f(XY)^n \right)\\
  =&\mathscr{D}\left(f(XY)\right)\\
  =&f(X).
 \end{align*}
For the second term, as $R(0,0)\neq 0$, $\frac{R'_Y(XY,Y)}{R(XY,Y)}$ is a power series in $XY$ and $Y$, so when we multiply this by $Y^2$ there is no diagonal term.
 
\end{proof}

\begin{proof}[Proof of Theorem \ref{main}]

Let the power series $Q(X,Y)$ be defined as $Q(X,Y)=P(X,Y)-Y$, then $Q'_Y(0,0)=P'_Y(0,0)-1\neq 0$, and $Q(X, f(X))=P(X,f(X))-f(X)=0$. 
According to Proposition \ref{furst}, 
\begin{align*}f &= \mathscr{D}\{Y^2 Q'_Y(XY,Y)/Q(XY,Y) \} 
\\ &=\mathscr{D}\{Y ^2 (P'_Y(XY,Y)-1 )/(P(XY,Y)-Y) \}
\\&=\mathscr{D}\{Y (1-P'_Y(XY,Y) )/(1-\frac{P(XY,Y)}{Y}) \}
\\&=\mathscr{D}\{Y (1-P'_Y(XY,Y) )(1+\sum _{m\geq 1} (\frac{P(XY,Y)}{Y})^m) \}.
\end{align*} 
We have the last equality due to the fact that $P'_Y(0,0)=0$, $\frac{P(XY,Y)}{Y}$ has no constant term, 
and therefore $1/(1-\frac{P(XY,Y)}{Y})=1+\sum_{m\geq 1}(\frac{P(XY,Y)}{Y})^m$.
\\As in each term of $Y(1-P'_Y(XY,Y))$ the power of $Y$ is larger than that of $X$, 
it cannot contribute to the diagonal.
Therefore, 

\begin{align*}
f_n & =[X^n Y^n] Y (1-P'_Y(XY,Y) )(1+\sum_{m\geq 1}(\frac{P(XY,Y)}{Y})^m)
\\ &=[X^n Y^n] Y (1-P'_Y(XY,Y) )(\sum_{m\geq 1}(\frac{P(XY,Y)}{Y})^m)
\\ &=\sum_{m\geq 1}[X^n Y^{m-1}] (1-P'_Y(X,Y))P(X,Y)^m.
\end{align*}

\end{proof}

\end{document}